\documentclass[11pt]{amsart}
\usepackage{amssymb, amsmath, amscd, amsthm, setspace, latexsym, vruler, enumerate}
\usepackage[toc,page]{appendix}
\title{Generators of maximal orders}
\author[R.V.~Kravchenko]{Rostyslav V.~Kravchenko}
\author{Marcin Mazur}
\author[B.V.~Petrenko]{Bogdan V.~Petrenko}

\address{University of Chicago} \email{rkchenko@gmail.com}

\address{
Department of Mathematics \\
Binghamton University \\
P.O. Box 6000 \\
Binghamton, NY 13902-6000, USA } \email{
mazur@math.binghamton.edu}

\address{
Department of Mathematics and Computer Science\\ Eastern Illinois University \\
600 Lincoln Avenue
\\ Charleston, IL 61920-3099
 }

\email{bvpetrenko@eiu.edu }

\newtheorem*{theorem*}{Theorem}
\newtheorem{theorem}{Theorem}[section]
\newtheorem{lemma}[theorem]{Lemma}
\newtheorem{proposition}[theorem]{Proposition}
\newtheorem{corollary}[theorem]{Corollary}
\newtheorem{remark}[theorem]{Remark}

\newtheorem{conjecture}[theorem]{Conjecture}

\newtheorem{example}[theorem]{Example}

\newtheorem{question}[theorem]{Question}

\def\GL{\text{\rm GL}}
\def\PGL{\text{\rm PGL}}

\def\M{\text{\rm M}}

\def\den{\text{\rm den}}
\def\g{\text{\rm g}}

\def\N{\text{\rm N}}

\newcommand{\Fg}{\mathbb{F}}

\newcommand{\mt}[1]{\M_{#1}(\Fg_q)}

\newcommand{\mw}[2]{\M_{#1}(\Fg_{#2})}

\newcommand{\F}[0]{\mathbb F}

\newcommand{\f}[1]{\frak{#1}}
\newtheorem*{te*}{Theorem}

\newcommand{\mspec}[0]{\text{\rm ~m-Spec~}}

\def\dsp{\def\baselinestretch{1.37}\large}

\begin{document}
\maketitle \dsp

\begin{abstract}
Let $R$ be the ring of algebraic integers in a number field $K$ and let $\Lambda$
be a maximal order in a semisimple $K$-algebra $B$.
Building on our previous work, we compute the smallest number of algebra generators
of $\Lambda$ considered as an $R$-algebra. This reproves and vastly extends the results of P.A.B.~Pleasants,
who considered the case when $B$ is a number field. In order to achieve our goal, we obtain
several results about counting generators of algebras which have finitely many elements. These
results should be of independent interest.

\vspace{4mm} \noindent {\bf Mathematics Subject Classification
(2010).} Primary 16H10, 16P10, 16S15, 15B33.
Secondary 11R45,  11R52.

\vspace{3mm} \noindent {\bf Keywords:} density, smallest number of
generators, probability of generating.
\end{abstract}

\tableofcontents

\section{Introduction}
The investigation of rings of algebraic integers that
can be generated by one element as a ring is an old topic in algebraic number theory. 
Such rings and their field of fractions
are often called {\bf monogenic}. In the early 1970s J. Browkin raised a more general problem of finding the smallest number
of ring generators for the ring of integers of a given number field. 
This question was answered by P. A. B. Pleasants \cite{pleasant}. The present work can be considered
as a vast generalization of the results of \cite{pleasant}, because we address the same question for maximal orders in finite dimensional semisimple $\mathbb Q$-algebras. 

In group theory, the smallest number of generators of a finite group and the structure of
generating sets have been studied for a long time, starting with the
seminal works by Hall \cite{hall}, Neumann \& Neumann \cite{nn} and Gash\"utz \cite{gas}.
The present paper is a continuation of our investigation in \cite{kmp}, where
we have developed methods to compute the smallest number of generators of algebras of finite type
over certain commutative rings. In particular, if $R$ is the ring of integers in a number field and 
$A$ is an $R$-algebra finitely generated as an $R$-module, then the computation of the smallest number of algebra generators  of $A$ often reduces to a computation, for each maximal ideal $\f m$ of $R$, of the smallest number
of generators of $A/\f m A$ as an $R/\f m$-algebra (see the beginning of Section 5 for exact statements).  With this in mind,
in Section 2 we develop techniques to count the number of generating tuples of $R$-algebras which have finitely many elements. 
This often reduces to the problem of counting generating tuples of matrix algebras over finite fields. This problem
is handled in Section 3. In Section 4, we review some background material about maximal orders in simple algebras.
Section 5 contains our main results, Theorem~\ref{smallest} and Theorem~\ref{smallest1}. These theorems
provide exact formulas for the smallest number of generators of a maximal $R$-order in a semisimple algebra over
a number field with the ring of integers $R$.

\paragraph*{\bf{Acknowledgments}}
This work was partially supported by a grant from the
Simons Foundation (\#245855 to Marcin Mazur).

%In \cite{kmp} we suggested that for a maximal order $A$ in a simple algebra knowing the smallest number $g_n$ of algebra generators of $A^n$ for all $n$ should tell 
%us a lot about the algebra $A$. As a consequence of the results in the present paper we show that this is not the case and there are many $A$'s with the
%same sequence $g_n$. 

\section{Generators of some finite dimensional algebras over fields}
The Jacobson radical of a ring $S$ is denoted by $J(S)$. 

\begin{lemma}\label{subalg}
Let $A=\prod_{i=1}^m A_i$ be a product of (left or right) Artinian algebras and let $B$ be a subalgebra of $A$ such that
\begin{enumerate}[\rm (i)]
\item for each $i$, the projections of $A$ onto $A_i$ maps $B$ onto $A_i$. 

\item $A=B+J(A)$. 
\end{enumerate}
Then $B=A$.
\end{lemma}
\begin{proof} 
Fix $i\in\{1,2,\ldots,m\}$. Let $e_i=(u_1,\ldots,u_m)$ be the element of $A$ such that
$u_i=1$ and $u_j=0$ for all $j\neq i$. Condition (ii) ensures that
$B$ contains an element $a=(a_1,\ldots,a_m)$ such that $a-e_i\in J(A)$. 
Thus $a_i$ is invertible in $A_i$ and $a_j\in J(A_j)$ for all $j\neq i$. 
By (i), $B$ contains an element $b=(b_1,\ldots,b_m)$ such that $b_i=a_i^{-1}$. Thus $ab=(c_1,\ldots,c_m)\in B$, where $c_i=1$ and $c_j\in J(A_j)$ for all $j\neq i$. Since $J(A_j)$ are nilpotent, $(ab)^N=e_i$ for $N$ sufficiently large.
Thus $B$ contains the idempotents $e_1,\ldots, e_m$. Condition (i) implies now that $B=A$.
\end{proof}

Recall that an Artinian algebra $A$ is called {\em primary} if $A/J(A)$ is
simple.

\begin{theorem}\label{primary}
Let $R$ be a commutative ring. 
Let $A=\prod_{i=1}^m A_i$ be a product of Artinian primary $R$-algebras, such that each $A_i$
is finitely generated as an $R$-module and such that the simple quotients of each $A_i$ are isomorphic
as $R$-algebras. Let $F=R/I$, where $I$ is the annihilator of the simple quotient  of each $A_i$. Let 
$k$ and $m$ be positive integers. Then 
elements 
\[a_1=(a_{11}, \ldots,  a_{1m}), \ldots  ,a_k=(a_{k1}, \ldots,  a_{km})\]
of $A$
generate $A$ as an $R$-algebra if and only if the following two
conditions are satisfied:
\begin{enumerate}[\rm (1)]
\item\label{it3} For any $i = 1, \ldots, m$, the elements $a_{1i},
\ldots, a_{ki}$ generate $A_i$ as an $R$-algebra. 

\item\label{it4}
For any $i \neq j$ there is no isomorphism $\Psi: A_j/J(A_j)\longrightarrow A_i/J(A_i)$
of $F$-algebras such that
\[a_{1i}+J(A_i) = \Psi( a_{1j}+J(A_j)), \ldots,  a_{ki}+J(A_i) = \Psi(a_{kj}+J(A_j)).\]
\end{enumerate}
\end{theorem}

Note that a special case of this theorem, when $J(A)=0$, is proved in \cite[Theorem~6.1]{kmp}.
Indeed, as we will see in the course of the proof of Theorem~\ref{primary}, $F$ is a field
and each quotient $A_i/J(A_i)$ is a finite dimensional
simple $F$-algebra. If $J(A)=0$, then $J(A_i)=0$ for all $i$ and $A$ can be considered as a product of $m$ copies of a finite dimensional simple $F$-algebra as in the Theorem~6.1 of \cite{kmp}. 

\begin{proof} 
Note that the center $Z_i$ of $A_i/J(A_i)$ is a field and $F$ can be identified with a subring
of $Z_i$ such that $A_i/J(A_i)$ is a finitely generated $F$-module. It follows that every
element of $A_i/J(A_i)$ is integral over $F$. In particular, the field $Z_i$ is integral over $F$
and therefore $F$ is also a field. Thus $A_i/J(A_i)$ is a finite dimensional simple $F$-algebra.
%Moreover, as $J$ is nilpotent, the ring $R$ is a commutative local Artinian ring.

We will use the fact that Theorem~\ref{primary} is true when $J=0$ (see \cite[Theorem~6.1]{kmp}).

The necessity of condition (1) is clear. If $a_1,\ldots, a_k$ generate the $R$-algebra $A$,
then the images of these elements in $A/J(A)$ generate $A/J(A)$ as an $F$-algebra. The necessity
of condition
(2) is now a simple consequence of the fact that Theorem~\ref{primary} is true for the $F$-algebra 
$A/J=\prod_{i=1}^{m}A_i/J(A_i)$.

Conversely, suppose that conditions (1) and (2) are satisfied. Then the images of $a_1,\ldots,a_k$
in $A/J(A)$ generate $A/J(A)$ as an $F$-algebra (we use again the fact that Theorem~\ref{primary} is true for the $F$-algebra $A/J(A)$). Let $B$ be the $R$-subalgebra of $A$ generated by
$a_1,\dots,a_k$. The quotient map $A\longrightarrow A/J(A)$ maps $B$ surjectively onto $A/J(A)$. Thus $B$ satisfies condition (ii) of Lemma~\ref{subalg}. Condition (1)
implies that $B$ also satisfies (i) of Lemma~\ref{subalg}. Thus $B=A$. \end{proof}

The following result is an algebra analog of a well known and very useful result of Gasch\"utz
\cite{gas} for groups. It is clear from the proof that analogs of this result hold for other
finite algebraic structures. 

\begin{theorem}\label{gaschutz}
Let $R$ be a commutative ring, let $A$ be an $R$-algebra of finite cardinality and let $f:A\longrightarrow
B$ be an epimorphism of $R$-algebras. For any natural number $k$ there is
a non-negative integer $g_k(A,B)$ with the following property:
any sequence of $k$ generators of the $R$-algebra $B$ can be lifted to exactly $g_k(A,B)$
sequences of $k$ generators of $A$. 
\end{theorem}

\begin{proof}
We fix $B$ and prove the result by induction on the order of $A$. If $|A|=|B|$ then $f$ is an isomorphism
and the result is clear with $g_k(A,B)=1$. Consider a surjective homomorphism $f:A\longrightarrow
B$ of $R$-algebras and let $h$ be the cardinality of the kernel of $f$. Let $\Phi$ be the family of all $R$-subalgebras of $A$ which are mapped onto $B$ by $f$.
Consider any $k$ elements $b_1,\ldots,b_k$ of $B$ which generate $B$ as an $R$-algebra. The number of $k$-tuples 
$a_1,\ldots,a_k$ of elements of $A$ such that $f(a_i)=b_i$ for $i=1,\ldots,k$, is equal to $h^k$. Any such
$k$-tuple generates an $R$-algebra in $\Phi$. By the inductive assumption, the number of such $k$-tuples
which generate $A$ is $h^k-\sum_{C\in \Phi, C\neq A}g_k(C,B)$.
\end{proof}

\begin{remark}
{\em Note that if $B$ can not be generated by $k$ elements, then the numbers $g_k(A,B)$ are not well
defined. It will however not matter what (finite) value we set for $g_k(A,B)$ in this case, and we
will use any convenient value (to keep our formulas uniform).
}
\end{remark}

\begin{proposition}
Let $A$ be an $R$-algebra with a nilpotent ideal $J$. If $B$ is a subalgebra of $A$ such that $B+J^2=A$
then $B=A$. 
\end{proposition}
\begin{proof}
Let $K=B\cap J$. Then $K$ is a nilpotent ideal of $B$ and $J=K+J^2$. We claim that $J^m=K^m+J^{m+1}$
for all positive integers $m$. Indeed, this holds for $m=1$ and if it holds for some $m$ then
\[J^{m+1}=J\cdot J^m=(K+J^2)(K^m+J^{m+1})=\]
\[=K^{m+1}+J^2\cdot K^m+K\cdot J^{m+1}+J^{m+3}\subseteq
K^{m+1}+J^{m+2}\subseteq J^{m+1}
\]
so the result holds for $m+1$. Thus our claim follows by induction. Since $J$ is nilpotent,
let $n$ be largest such that $J^n\neq 0$. Then $J^n=K^n$ and by backward induction we conclude
that $J^m=K^m$ for all $m$. Thus $J=K$, so $J\subseteq B$ and $A=B+J^2\subseteq B$. 
\end{proof}

As an immediate corollary we get the following result.

\begin{corollary}\label{lifts}
If $A$ is an Artinian $R$-algebra and $b_1,\ldots,b_k$ generate the $R$-algebra $A/J(A)^2$ then
any lifts $a_1,\ldots, a_k$ of $b_1,\ldots, b_k$ to $A$ generate $A$ as an $R$-algebra.
In particular, if $A$ is finite then $g_k(A, A/J(A)^2)=|J(A)^2|^k$ for every positive integer $k$. 
\end{corollary}

\begin{theorem}\label{bisimple}
Let $A$ be a finite dimensional algebra over a field $F$ such that $J(A)^2=0$ and $J(A)$ is
simple as a $A/J(A)$-bimodule. Suppose furthermore that $A/J(A)$ is a separable $F$-algebra.
For $k\geq 2$, any sequence of $k$ elements which generate the $F$-algebra $A/J(A)$ can be lifted
to a sequence of $k$ elements in $A$ which generate $A$ as an $F$-algebra. If $A$ is commutative, the same holds for $k=1$. 

Suppose furthermore that $F$ is finite and $J(A)\neq 0$, and let $C$ be the intersection of $J(A)$ and the center of $A$.
Then $g_k(A, A/J(A))= |J(A)|^k- |J(A)/C|$.  
\end{theorem}
\begin{proof}
The result is obvious when $J(A)=0$, so we assume that $J(A)\neq 0$.
Recall the following fundamental result of Wedderburn and Malcev (see \cite{pierce}, Chapter 11)):

\begin{theorem}\label{split}
Let $F$ be field and $T$ a finite dimensional $F$-algebra with Jacobson radical $J$ such that $T/J$
is a separable $F$-algebra.
Then $T$ has a subalgebra $S$ such that $T=S\oplus J$. Moreover, any two such subalgebras are conjugate by a unit of the form $1-u$ for some $u\in J$. 
\end{theorem}
Choose a subalgebra $S$ of $A$ such that $A=S\oplus J(A)$. If $B$ is any subalgebra of $A$ such that 
$A=B+J(A)$ then $J(B)=B\cap J(A)$ and $J(B)$ is a $B/J(B)$-bimodule. Thus $A/J(A)=B/J(B)$ and $J(B)$
is a $A/J(A)$-bi-submodule of $J(A)$. Since $J(A)$ is simple, either $J(B)=0$ 
or $J(B)=J(A)$. In the latter case we have $A=B$, and in the former case $B=(1-u)S(1+u)$
for some $u\in J(A)$. 

Let $x_1,\ldots, x_k$ be elements of $A/J(A)$ which generate $A/J(A)$ as an $F$-algebra.
Let $s_i$ be the unique element of $S$ such that $x_i=s_i+J(A)$. Any lift of $x_1,\ldots, x_k$
to $A$ is of the form $s_1+u_1,\ldots, s_k+u_k$ for some $u_i\in J(A)$. 
We will show that $s_1+u_1,\ldots, s_k+u_k$ do not generate $A$ as an $F$-algebra if and only if there is $u\in J(A)$ such that $u_i=[s_i,u]$ for $i=1,\ldots,k$.  
Let $B$ be the $F$-subalgebra
of $A$ generated by $s_1+u_1,\ldots, s_k+u_k$. Then $A=B+J(A)$. If $B\neq A$ then
$B=(1-u)S(1+u)$ for some $u\in J(A)$.
It follows that $s_i+u_i=(1-u)s_i(1+u)=s_i+[s_i,u]$ (as both sides are elements of $B$ which lift $x_i$). Conversely, if $u_i=[s_i,u]$ for $i=1,\ldots,k$ and some $u\in J(A)$, then $B=(1-u)S(1+u)$
so $B\neq A$. 

Consider the map $\phi: J(A)\longrightarrow \prod_{i=1}^kJ(A)$ given by $\phi(u)= ([s_1,u],\ldots,[s_k,u])$. 
This is an $F$-linear map so it can not be onto when $k\geq 2$. When $A$ is abelian then this map
is trivial, so it is not onto even for $k=1$. This means that we can always choose $u_i\in J(A)$,
$i=1,\ldots,k$ such that $s_1+u_1,\ldots, s_k+u_k$ generate $A$ as an $F$-algebra. In other words,
$x_1,\ldots, x_k$ can be lifted to a sequence generating $A$ as an $F$-algebra.

Suppose now that $F$ is finite. Then the number of lifts of $x_1,\ldots, x_k$ to $A$
is equal to $|J(A)|^k$. The number of lifts which do not generate $A$ is equal to the cardinality
of the image of $\phi$. Note that the kernel of $\phi$ is exactly $C$, so the image of $\phi$
has cardinality $|J(A)/C|$. Thus $g_k(A, A/J(A))= |J(A)|^k- |J(A)/C|$.
\end{proof}

\section{Matrix rings over finite fields}
As we have seen in the previous section, computations of the number of generating tuples
of an algebra $A$ with finitely many elements can often be reduced to same computations for $A/J(A)$.
This in turn reduces to computations of generating tuples for finite simple algebras.
Any such an algebra is of the form $\M_n(F)$, where $F$ is a finite field.   
In this section we extend some results about generators of matrix
rings over finite fields obtained in \cite{kmp}. Let us first recall some of the results
from \cite{kmp}. Note that our notation here will differ slightly from the one introduced in \cite{kmp}. 

For a finite $R$-algebra $A$ and a positive integer $k$ we denote by $\g_k(A,R)$ the number
of $k$-tuples of elements in $A$ which generate $A$ as an $R$-algebra. If there is no ambiguity
about $R$ we simply write $\g_k(A)$. Recall that for a quotient $B$ of $A$ we defined $g_k(A,B)$
as the number of ways a $k$-tuple generating the $R$-algebra $B$ can be lifted to a $k$-tuple 
generating the $R$-algebra $A$ (note that we use an italic $g$ here, and not a roman $\g$). If there is a need to indicate $R$ explicitly, we will 
write $g_k(A,B,R)$ for $g_k(A,B)$. Thus $\g_k(A,R)=\g_k(B,R)g_k(A,B,R)$. To simplify
the notation, we set 
$\g_{k}(n,q,r)=\g_k(\mw{n}{q^r},\mathbb F_q)$ and $\g_{k}(n,q)=\g_{k}(n,q,1)$.
We can now state Theorem~6.3. from \cite{kmp}:

\begin{theorem}\label{formula}
Let $A=\M_n(\mathbb F_{q^s})$. Then $A^m$ can be generated
by $k$ elements as an $\mathbb F_{q}$-algebra iff
\[ m\leq \frac{\g_k(n,q,s)}{s|\text{\rm PGL}_n(\mathbb F_{q^s})|}.\]
Furthermore,
\[ \g_k(A^m,\mathbb F_{q})=\prod_{i=0}^{m-1}(\g_k(n,q,s)-i\cdot s\cdot |\text{\rm PGL}_n(\mathbb F_{q^s})|).\]
\end{theorem}

The following theorem summarizes the results of section 7 of \cite{kmp}.

\begin{theorem}\label{exactformula}
\begin{enumerate}[\rm (i)]
\item $\g_k(1,q)=q^k$.

\item $\g_{k}(2,q)=q^{2k+1}(q^{k-1}-1)(q^k-1)$.

\item $\g_{k}(3,q)= q^{3k+4}(q^{k-1}-1)(q^{k-1}+1)(q^k-1)(q^{3k-2}+q^{2k-2}-q^{k}-2q^{k-1}-q^{k-2}+q+1)$.

\item $\g_{k}(n,q)\geq q^{kn^2}-2^{\frac{n+6}{2}}q^{n^2k-(k-1)(n-1)}$.
\end{enumerate}
\end{theorem}
The computation of the numbers $\g_k(n,q)$ for $n\geq 4$ remains one of the main outstanding
problems in this area. The main result of this section expresses the numbers
$\g_{k}(n,q,r)$ in terms of the numbers $\g_k(n,q)$. We start with the following proposition.

\begin{proposition}\label{fgen}
Let $A$ be a subring of $\mt{n}$. Then $\mathbb F_q A=\mt{n}$ if and only if $A$ is conjugate
to $\mw{n}{\pi}$ for some subfield $\mathbb F_{\pi}$ of $\mathbb F_q$.
\end{proposition}
\begin{proof}
Let $Z$ be the center of $A$. We claim that $Z\subseteq \mathbb F_q$. Indeed,
the assumption that $\mathbb F_q A=\mt{n}$ implies that any element central in $A$
remains central in $\mt{n}$. Thus $Z= \mathbb F_{\pi}$ is a subfield of $\mathbb F_q$.
Let $J$ be the Jacobson radical of $A$. Then $J$ is nilpotent and $\mathbb F_q J$ is a nilpotent
ideal of $\mt{n}$. It follows that $\mathbb F_q J=0$, and consequently $J=0$. This proves
that $A$ is semisimple. Since the center of $A$ is a field, $A$ is simple.  Thus $A$ is isomorphic
as $\mathbb F_{\pi}$-algebra
to $\mw{r}{\pi}$ for some $r$. The embedding $A\hookrightarrow \mt{n}$ extends
to an $\mathbb F_q$-algebra homomorphism $f: A\otimes_{\mathbb F_{\pi}}\mathbb F_q\longrightarrow \mt{n}$,
which is surjective by the assumption that $\mathbb F_q A=\mt{n}$. On the other hand,
$A\otimes_{\mathbb F_{\pi}}\mathbb F_q\simeq \mw{r}{q}$ is a central simple $\mathbb F_q$-algebra,
so $f$ is injective. It follows that $f$ is an isomorphism and $r=n$. There is an $\mathbb F_{\pi}$-algebra
isomorphism $g:A\longrightarrow \mw{n}{\pi}$ which extends to an $\mathbb F_q$-algebra isomorphism
$g^*: A\otimes_{\mathbb F_{\pi}}\mathbb F_q\longrightarrow \mt{n}$. By the Noether-Skolem theorem,
there is an invertible matrix $u\in \mt{n}$ such that $f(x)=ug^{*}(x)u^{-1}$ for all
$x\in A\otimes_{\mathbb F_{\pi}}\mathbb F_q$. It follows that $A=u\mw{n}{\pi}u^{-1}$.

Conversely, if $A=u\mw{n}{\pi}u^{-1}$ then clearly $\mathbb F_q A=\mt{n}$.
\end{proof}

\begin{theorem}\label{mobius}
\[\frac{\g_{m}(n,q,r)}{|\PGL_n(\mathbb F_{q^r})|}=
\sum_{s|r}\frac{\g_{m}(n,q^s)}{|\PGL_n(\mathbb F_{q^s})|}\mu(r/s).
\]
\end{theorem}
\begin{proof}
Note that
an $m$-tuple of elements in $\mw{n}{q^r}$ generates  $\mw{n}{q^r}$ as an $\mathbb F_{q^r}$-algebra
if and only the $\mathbb F_q$-algebra $A$ it generates satisfies $\mathbb F_{q^r}A=\mw{n}{q^r}$.
Let $\Phi$ be the collection of all $\mathbb F_q$-subalgebras $A$ of $\mw{n}{q^r}$
such that $\mathbb F_{q^r}A=\mw{n}{q^r}$.  Thus
\[\g_{m}(n,q^r)=\sum_{A\in \Phi}\g_m(A,\mathbb F_q).
\]
By Proposition~\ref{fgen}, $A\in\Phi$ if and only if $A$ is conjugate to $\mw{n}{q^s}$ for some divisor $s$ of $r$.
Note that if $C\in \GL_n(\mathbb F_{q^r})$ satisfies $C\mw{n}{q^s}C^{-1}=\mw{n}{q^s}$, then
conjugation by $C$ induces $\mathbb F_{q^s}$-algebra automorphism of $\mw{n}{q^s}$. Since all such automorphisms
are inner by the Noether-Skolem theorem, there is $D\in\GL_n(\mathbb F_{q^s})$ such that $D^{-1}C$ centralizes
$\mw{n}{q^s}$. It follows that $D^{-1}C\in \mathbb F_{q^r}$. Consequently, the number of
matrices $C$ which stabilize $\mw{n}{q^s}$ is equal to $|\GL_n(\mathbb F_{q^s})|(q^r-1)/(q^s-1)$. The number of
elements in $\Phi$ which are conjugate to $\mw{n}{q^s}$ is then equal to
$\displaystyle \frac{|\GL_n(\mathbb F_{q^r})|(q^s-1)}{|\GL_n(\mathbb F_{q^s})|(q^r-1)}$. Note that if $A$ is conjugate
to $\mw{n}{q^s}$ then $\g_m(A,\mathbb F_q)=\g_{m}(n,q,s)$. Thus we get the following formula.
\[\g_{m}(n,q^r)=\sum_{s|r}\g_{m}(n,q,s)\frac{|\GL_n(\mathbb F_{q^r})|(q^s-1)}{|\GL_n(\mathbb F_{q^s})|(q^r-1)}.
\]
Thus
\[\frac{\g_{m}(n,q^r)}{|\PGL_n(\mathbb F_{q^r})|}=
\sum_{s|r}\frac{\g_{m}(n,q,s)}{|\PGL_n(\mathbb F_{q^s})|}
\]
By M\"obius inversion formula, we get
\[\frac{\g_{m}(n,q,r)}{|\PGL_n(\mathbb F_{q^r})|}=
\sum_{s|r}\frac{\g_{m}(n,q^s)}{|\PGL_n(\mathbb F_{q^s})|}\mu(r/s).
\]
\end{proof}

We will need the following technical result.

\begin{lemma}\label{estimat}
For any positive integers $r,t,n$ such that $n\neq 2$ there is a constant $c(n,r,t)$ such that 
\[ \g_2(\M_{n}(\F_{q^r})^{t}, \F_q)\geq q^{2n^2rt}\left(1-\frac{25t}{q^2} \right)^t\]
for all $q\geq c(n,r,t)$.
\end{lemma}
\begin{proof}
When $n=1=r$,  Theorem~\ref{formula} and Theorem~\ref{exactformula}(i) yield  
\[ \g_2((\F_{q})^{t}, \F_q)=\prod_{i=0}^{t-1}(q^2-i)\geq q^{2t}\left(1-\frac{25t}{q^2} \right)^t.
\] 
Suppose now that either $n\geq 3$ or $n=1$ and $r\geq 2$. By Theorem~\ref{mobius},
\[ \g_2(n,q,r)\geq \g_2(n,q^r)-\sum_{s|r, s<r}q^{2sn^2}\frac{|\PGL_n(\mathbb F_{q^r})|}{ |\PGL_n(\mathbb F_{q^s})|}.\]
Note that 
\[ \frac{a^y-1}{a^x-1}\leq \frac{y}{x}a^{y-x}\]
for any $a>1$ and $0<x<y$ (for a simple proof use the Mean Value Theorem).
Using this inequality, we get
\[\frac{|\PGL_n(\mathbb F_{q^r})|}{|\PGL_n(\mathbb F_{q^s})|}=\frac{q^s-1}{q^r-1}\prod_{i=0}^{n-1}
\frac{q^{rn}-q^{ri}}{q^{sn}-q^{si}}\leq \prod_{i=0}^{n-1}q^{(r-s)i}\prod_{i=0}^{n-2}
\frac{r}{s}q^{(r-s)(n-i)}=\left(\frac{r}{s}\right)^{n-1}q^{(r-s)(n^2-1)}.\]
Since $1\leq s\leq r/2$, we have
\[\g_2(n,q,r)\geq \g_2(n,q^r)-r^{n}q^{\frac{3}{2}r(n^2-1)+r}.\]
When $n=1$, $r>1$, we get by Theorem~\ref{exactformula}(i) that 
\[ \g_2(1,q,r)\geq q^{2r}-rq^{r}\]
and using Theorem~\ref{formula} we conclude that
\[ \g_2((\F_{q^r})^{t}, \F_q)=\prod_{i=0}^{t-1}(\g_2(1,q,r)-ir)\geq q^{2rt}\left(1-\frac{2r}{q^r}  \right)^t\geq q^{2rt}\left(1-\frac{4}{q^2}  \right)^t\geq q^{2rt}\left(1-\frac{25t}{q^2}  \right)^t\]
for $q>\max(25,t,r)$.

Suppose now that $n\geq 3$. By Theorem~\ref{exactformula}(iv) we have
\[\g_2(n,q,r)\geq q^{2rn^2}-2^{(n+6)/2}q^{(2n^2-n+1)r}-r^{n}q^{\frac{3}{2}r(n^2-1)+r}.\]
Note that $|\PGL_n(\mathbb F_{q^r})|\leq q^{r(n^2-1)}$. Hence, when $q>\max(tr,r^n) $ and $0\leq i<t$
we have
\[\g_2(n,q,r)-ir|\PGL_n(\mathbb F_{q^r})|\geq q^{2rn^2}-2^{(n+6)/2}q^{(2n^2-n+1)r}-r^{n}q^{\frac{3}{2}r(n^2-1)+r}-irq^{r(n^2-1)}\geq\]
\[\geq q^{2rn^2}\left(1-\frac{2^{9/2}}{q^{2r}}\left(\frac{2^{1/2}}{q^r}\right)^{n-3}-
\frac{1}{q^{2r}}\frac{r^n}{q^{\frac{1}{2}r(n^2-3)}}-\frac{1}{q^{2r}}\frac{ir}{q^{r(n^2-1)}}     \right) \geq q^{2rn^2}\left(1-\frac{25t}{q^2}\right).
\]
Using Theorem~\ref{formula} we see that 
\[\g_2(\M_{n}(\F_{q^r})^{t}, \F_q)\geq q^{2n^2rt}\left(1-\frac{25t}{q^2} \right)^t \]
provided $q>\max(tr,r^n, 25)$.
\end{proof}

\section{Maximal orders in simple algebras}
In this section we review some facts about orders in simple algebras. All the results and their proofs can be found in the excellent book \cite{reiner}.

Let $R$ be a Dedekind domain with field of fractions $K$ and let $H$ be a central simple
$K$-algebra. Let $\Lambda$ be a maximal order in $H$ and let $P$ be a maximal ideal of $R$.
Our goal is to get some understanding of algebras of the form $\Lambda/P^t\Lambda$.
Let us consider the $P$-completions $\Lambda_{P}$, $H_{P}$, $K_{P}$, $R_{P}$ of $\Lambda$,
$H$, $K$ and $R$ respectively. Then $\Lambda_{P}$ is a maximal order in the central
simple $K_{P}$-algebra $H_{P}$. $R_P$ is a complete discrete valuation ring with maximal ideal $\hat{P}$ and $\Lambda/P^t\Lambda=\Lambda_P/\hat{P}^t\Lambda_P$. 

There are a positive integer $r_P=r$, a finite dimensional central division
algebra $D$ over $K_P$ with the unique maximal order $\Delta$, and an isomorphism
$H_P\cong\M_r(D)$ which identifies $\Lambda_P$ with $\M_r(\Delta)$. We call $r_P$
the {\bf local capacity} of $H$ at $P$. The index of $D$ is denoted by $m_P=m$ and
called the {\bf local index} of $H$ at $P$ (so the dimension of $D$ over $K_P$ is $m_P^2$). 
We say that $P$ is {\bf ramified} in $H$ if $m_P>1$, and {\bf unramified} if $m_P=1$.
Furthermore, there is an element $\pi\in \Delta$ such that
ideals in $\Delta$ are exactly the subsets of the form $\pi^j\Delta$, $j=1,2,\ldots$ (so, in particular, each one-sided ideal in $\Delta$ is two-sided). Now $\hat{P}\Delta=\pi^e\Delta$,
where $e=e_P$ is called the {\bf ramification index} of $D$ over $K_P$. The residue ring
$\overline{\Delta}=\Delta/\pi\Delta$ is a division algebra over the field $R_P/\hat{P}\cong R/P$. 
The dimension of $\overline{\Delta}$ over $R/P$ is denoted by $f=f_P$ and called the {\bf inertial}
degree of $D$ over $K_P$. Note that $e_Pf_P=m_P^2$ (also $e_P|m_P$ and $m_P|f_P$). 

Suppose in addition that $R/P\cong \mathbb F_q$ is a finite field. Then $e_P=f_P=m_P$ and $\overline{\Delta}\cong \mathbb F_{q^m}$. Moreover, $\Delta$ contains a primitive $(q^m-1)$th root
of 1, denoted $\omega$. Given any generator of $\hat{P}$, one can choose $\pi$ so that $\pi^m\in \hat{P}$ is the given generator and $\pi\omega=\omega^{q^s}\pi$, where $s$ is a positive integer
independent of all the choices (of $\omega$, $\pi$, etc.) and such that $1\leq s \leq m$, $(s,m)=1$.
We have $\Delta=R[\omega, \pi]$. The fraction $s/m$ is called the {\bf Hasse invariant} of $D$.

\begin{theorem}\label{simple}
Let $R$ be a Dedekind domain with field of fractions $K$ and let $H$ be a central simple
$K$-algebra. Let $\Lambda$ be a maximal order in $H$ and let $P$ be a maximal ideal of $R$.
Let $t$ a positive integer and let $A=\Lambda/P^t\Lambda$. We use the notation introduced above. Then

\begin{enumerate}[{\rm (i)}]
\item $A$ is an Artinian primary $R$-algebra, finitely generated as an $R$-module, and $A/J(A)\cong \M_r(\overline{\Delta})$, where $r$ is the local capacity of $H$ at $P$.

\item $J(A)/J(A)^2$ is simple as $A/J(A)$-bimodule. 

\item Suppose in addition that $R/P\cong \mathbb F_q$ is finite. Then  $A/J(A)\cong \M_r(\mathbb F_{q^m})$, where $m$ is the local index of $H$ at $P$. Moreover, if
\[ C:=\{x\in J(A)/J(A)^2: ax=xa \ \text{for all}\ a\in A/J(A)\} \] 
then
\[ |C|=\begin{cases} 1,& \text{if}\  m>1;\\
q,& \text{if} \ m=1.
\end{cases}\]
\end{enumerate}
\end{theorem}

\begin{proof}
We use the notation and results discussed at the beginning of this section. In particular,
we identify $\Lambda_P$ with $\M_r(\Delta)$.
We start with 
observing that two sided ideals of $\Lambda_P=\M_r(\Delta)$ are of the form $(\pi\Lambda_P)^n=\pi^n\Lambda_P=\M_r(\pi^n\Delta)$, $n\in \mathbb N$. 
Also, $\hat{P}\Lambda_P=\pi^e \Lambda_P$ and $\Lambda_P/\pi \Lambda_P\cong \M_r(\overline{\Delta})$.
It is clear now that $A\cong\Lambda_P/\hat{P}^t \Lambda_P=\Lambda_P/\pi^{et} \Lambda_P$ is an Artinian $R_P$-algebra with Jacobson radical $J(A)=\pi \Lambda_P/\pi^{et} \Lambda_P$ and
therefore $A/J(A)\cong \M_r(\overline{\Delta})$.

Note now that conjugation by $\pi$ ($x\mapsto \pi^{-1} x\pi$) defines an automorphism of 
$\Lambda_P$ (and of $\Delta$) which induces an automorphism $\phi$ of $\M_r(\overline{\Delta})$.
We may identify $J(A)/J(A)^2= \pi \Lambda_P/\pi^{2} \Lambda_P$ with $\M_r(\overline{\Delta})$
by sending the class of an element $\pi g$, $g\in \Lambda_P$ to the image of $g$ in $\M_r(\overline{\Delta})$. This yields an isomorphism of $A/J(A)$-bimodules, where
the $\M_r(\overline{\Delta})$-bimodule structure on $\M_r(\overline{\Delta})$ 
is given by $u\cdot x\cdot v=\phi(u)xv$ (where on the right hand side we use the ring
multiplication in $\M_r(\overline{\Delta})$). It is clear now that $J(A)/J(A)^2$ is simple as $A/J(A)$-bimodule. 

Finally, suppose that $R/P\cong \mathbb F_q$ is finite. Then $\overline{\Delta}\cong \mathbb F_{q^m}$
and therefore $A/J(A)\cong \M_r(\mathbb F_{q^m})$. Choosing a primitive $(q^m-1)$th root
of unity $\omega$ and $\pi$ so that $\pi\omega=\omega^{q^s}\pi$, we see that in this case the automorphism $\phi$ is induced on $\M_r(\mathbb F_{q^m})$ by the automorphism $F^{-s}$ of $\mathbb
F_{q^m}$, where $F$ is the Frobenius automorphism of $F_{q^m}$ over $F_{q}$. Under our identifications, $C$ is identified with $\{x\in  \M_r(\mathbb F_{q^m}): \phi(u)x=xu \ 
\text{for all}\ u\in  \M_r(\mathbb F_{q^m})\}$. When $m>1$, $\phi$ is a non-trivial automorphism
and it is easy to see that $C=\{0\}$ in this case. When $m=1$ then $\phi$ is the identity and
$C$ is identified with the center of $\M_r(\mathbb F_{q})$, which has $q$ elements. 
\end{proof}

\section{Maximal orders in semisimple algebras}
We start this section by recalling a few results from \cite{kmp}. For a commutative ring $R$
and an $R$-algebra $A$ which is finitely generated as an $R$-module we write $r(A,R)$
for the smallest number $k$ such that $A$ can be generated by $k$ elements as an $R$-algebra. 
Let $R$ be a Dedekind domain
with the field of fractions $K$. For a maximal ideal $\f p$ of $R$ define 
$r_{\f p}(A)=r(A/\f p A, R/\f p)$. Let $r_K(A)=r(A\otimes_R K, K)$.  
The following result is a special case of Theorem~5.7. in \cite{kmp}.

\begin{theorem}\label{localglobal}
Let $R$ be a Dedekind domain
with the field of fractions $K$ and let $A$ be an $R$-algebra which is finitely generated as an $R$-module. If $k\geq r_{\f p}(A)$ for all maximal ideals $\f p$ of $R$ and $k\geq r_K(A)+1$
then $A$ can be generated by $k$ elements as an $R$-algebra.
\end{theorem}

Suppose now that $R$ is the ring of algebraic integers in a number field $K$. Then,
for each maximal ideal $\f p$ of $R$, the residue field $R/\f p$ is finite with $\N(\f p)$
elements. For a positive integer $k$, let $\g_k(A,\f p)=\g_k(A/\f p A, R/\f p)$. 
Suppose that $A$ is a free $R$-module of finite rank. In \cite{kmp} we introduced the notion
of density $\den_k(A)$ of the set of all $k$-tuples in $A^k$ which generate $A$ as and $R$-algebra.
Roughly speaking, a choice of an integral basis of $R$ and of a basis of $A$ over $R$ allows us to introduce integral coordinates on all cartesian
powers $A^k$, $k\in \mathbb N$. For any
subset $S$ of $A^k$ and any $N$ we consider the finite set $S(N)$ of all points whose coordinates are in the interval $[-N,N]$. We define the density $\den(S)$ of $S$ as the limit
$\displaystyle \lim_{N\to\infty} \frac{|S(N)|}{|A^k(N)|}$ (we do not claim that it always exists).

\begin{theorem}\label{density}
Let $R$ be the ring of algebraic integers in a number field $K$ and let $A$ be an $R$-algebra
which is free of rank $m$ as an $R$-module. Then
\begin{equation}\label{dens1}
 \den_k(A)=\prod_{\f p\in \mspec R}\frac{\g_k(A,\f p)}{\N(\f p)^{mk}}
\end{equation}
for any positive integer $k$, where $\mspec R$ is the set of all maximal ideals of $R$.
\end{theorem}
Note in particular that if $\den_k(A)>0$ then $A$ can be generated by $k$ elements as an $R$-algebra.
The right hand side of (\ref{dens1}) makes sense when $A$ is a projective $R$-module. But so far 
we have no good definition of density in this case. Nevertheless, we propose the following conjecture.

\begin{conjecture}\label{conjdense}
If $A$ in Theorem~\ref{density} is finitely generated and projective as an $R$-module and
if for some $k$ the right hand side of (\ref{dens1}) is positive, then $A$ can be generated by
$k$ elements as an $R$-algebra.  
\end{conjecture}

Finally, we state the following result, which is Proposition~2.12. in \cite{kmp}.

\begin{proposition}\label{productgen}
Suppose that $A=\prod_{i=1}^s A_i$ is a product of $R$-algebras $A_1,\ldots,A_s$ such that
for any maximal ideal $\f m$ of $R$ and any $i\neq j$ the $R/\f m$-algebras
$A_i\otimes_{R} R/\f m$ and $A_j\otimes_{R} R/\f m$ do not have isomorphic quotients.
Then elements 
\[a_1=(a_{11}, \ldots,  a_{1s}), \ldots  ,a_k=(a_{k1}, \ldots,  a_{ks})\]
of $A$ generate $A$ as an $R$-algebra if and only if 
the elements $a_{1i},
\ldots, a_{ki}$ generate $A_i$ as an $R$-algebra for $i = 1, \ldots, s$. 
\end{proposition}

Let now $R$ be a Dedekind domain with the field of fractions $K$. Let $H$ be a separable
$K$-algebra and let $\Lambda$ be a maximal $R$-order in $H$. Our goal is to apply the above results in order to find  the smallest number of generators of $\Lambda$ as an $R$-algebra. To this end we
introduce the following notation.

{\sffamily{
\begin{itemize}
%\begin{enumerate}[{\rm (i)}]
\item $R$ is a Dedekind domain with the field of fractions $K$.

\item $P$ is a maximal ideal of $R$.

\item $H$ is a separable $K$-algebra and $\Lambda$ is a maximal $R$-order in $H$.

\item $H=\prod_{i} H_i$, where $H_i$ are simple separable $K$-algebras.

\item $K_i$ is the center of $H_i$ and $R_i$ is the integral closure of $R$ in $K_i$.

\item $\Lambda= \prod_{i} \Lambda_i$, where $\Lambda_i$ is a maximal $R_i$-order in $H_i$.

\item $PR_i=\prod_{j}P_{i,j}^{e_{i,j}}$, where $P_{i,j}$ are maximal ideals of $R_i$ and $e_{i,j}$ are positive integers.

\item $A_{i,j}=\Lambda_{i}/P_{i,j}^{e_{i,j}}\Lambda_i$.
%\end{enumerate}
\end{itemize}}}
It follows that
\begin{equation}\label{modP}
 \Lambda/P\Lambda= \prod_{i} \Lambda_i/P\Lambda_i= \prod_{i,j}\Lambda_{i}/P_{i,j}^{e_{i,j}}\Lambda_i
 =\prod_{i,j}A_{i,j}.
\end{equation}

%We have $H=\prod_{i} H_i$, where $H_i$ are simple separable $K$-algebras.
%Let $K_i$ be the center of $H_i$ and let $R_i$ be the integral closure of $R$ in $K_i$. Then $\Lambda= \prod_{i} \Lambda_i$, where $\Lambda_i$
%is a maximal $R_i$-order in $H_i$. Let $P$ be a maximal ideal of $R$. Then $PR_i=\prod_{j}P_{i,j}^{e_{i,j}}$, where $P_{i,j}$ are maximal ideals of $R_i$ and $e_{i,j}$ are
%positive integers. It follows
%that
%\[ \Lambda/P\Lambda= \prod_{i} \Lambda_i/P\Lambda_i= \prod_{i,j}\Lambda_{i}/P_{i,j}^{e_{i,j}}\Lambda_i.
%\]

We now state our first result of this section.

\begin{theorem}\label{3gen}
Let $R$ be a Dedekind domain whose residue fields are infinite. Let $K$ be the field of fractions
of $R$ and let $H$ be a separable $K$-algebra. Let $\Lambda$ be a maximal $R$-order in $H$.
Suppose that for each maximal ideal $P$ of $R$ the quotient of $\Lambda/P\Lambda$ by its
Jacobson radical is a separable $R/P$-algebra (note that this is true when all residue fields
of $R$ are perfect). Then $\Lambda$ can be generated by 3 elements as an $R$-algebra.
\end{theorem}

\begin{proof}
We proved in \cite{mmbp} that separable algebras over infinite fields are $2$-generated.
Thus, $r_K(\Lambda)\leq 2$. By Theorem~\ref{localglobal}, it suffices now to show
that for every maximal ideal $P$ of $R$ the $R/P$-algebra $\Lambda/P\Lambda$ is $2$-generated.
To this end, let $P$ be a maximal ideal of $R$. We use the notation introduced above. 

%Recall that $H=\prod_{i} H_i$ and $\Lambda= \prod_{i} \Lambda_i$, where, for each $i$, $H_i$ is a simple separable $K$-algebras with center $K_i$, $\Lambda_i$ is a maximal $R_i$-order in $H_i$
%and $R_i$ is the integral closure of $R$ in $K_i$ (note that $K_i$ is finite separable over $K$ and
%$R_i$ is a Dedekind domain). We have $PR_i=\prod_{j}P_{i,j}^{e_{i,j}}$, where $P_{i,j}$ are maximal ideals of $R_i$ and $e_{i,j}$ are positive integers. Thus
%\[ \Lambda/P\Lambda= \prod_{i} \Lambda_i/P\Lambda_i= \prod_{i,j}\Lambda_{i}/P_{i,j}^{e_{i,j}}\Lambda_i.
%\]
%Let $A_{i,j}=\Lambda_{i}/P_{i,j}^{e_{i,j}}\Lambda_i$. 
By Theorem~\ref{simple}, each $A_{i,j}$ in (\ref{modP}) is an 
Artinian primary $R_i/P_{i,j}^{e_{i,j}}$-algebra, finitely generated as $R_i/P_{i,j}^{e_{i,j}}$-module. Note that $R_i/P_{i,j}^{e_{i,j}}$ is a finite dimensional
$R/P$-algebra. It follows that $A_{i,j}$ is a finite dimensional primary $R/P$-algebra.
Furthermore, $A_{i,j}/J(A_{i,j})$ is a separable $R/P$-algebra, hence it is $2$-generated.
By Theorem~\ref{simple}, $J(A_{i,j})/ J(A_{i,j})^2$ is a simple  $A_{i,j}/ J(A_{i,j})$-bimodule.
Thus any pair of $R/P$-algebra generators of $A_{i,j}/ J(A_{i,j})$ can be lifted to a pair
of $R/P$-algebra generators of $A_{i,j}/ J(A_{i,j})^2$ by Theorem~\ref{bisimple}.
By Corollary~\ref{lifts}, any pair of $R/P$-algebra generators of $A_{i,j}/ J(A_{i,j})^2$ can be lifted to a pair of $R/P$-algebra generators of $A_{i,j}$. 
It follows that any pair of $R/P$-algebra generators of $A_{i,j}/ J(A_{i,j})$ can be lifted to a pair of $R/P$-algebra generators of $A_{i,j}$.

Now let $B$ be one of the simple quotients of $\Lambda/P\Lambda$ and let $A_B$ be the product
of all those $A_{i,j}$  which have simple quotient isomorphic to $B$ as an $R/P$-algebra.
Then $\Lambda/P\Lambda=\prod_B A_B$. By Proposition~\ref{productgen}, it suffices
to show that each $R/P$-algebra $A_B$ is $2$-generated (recall that $R/P$ is an infinite field). As $A_B/J(A_B)$ is a separable
$R/P$-algebra, it is $2$-generated. It is now clear that Theorem~\ref{primary} implies
that any pair of generators of $A_B/J(A_B)$ lifts to a pair of generators of $A_B$.
\end{proof}

\begin{question}
In the notation of Theorem~\ref{3gen}, is $\Lambda$ generated by two elements as an $R$-algebra?
\end{question}

\noindent
We expect that the answer is negative in general, but at present we do not have a counterexample. Consider however the following related question. Suppose that $R$ is a Dedekind domain
with fraction field $K$ and with all residue fields infinite and perfect. Let $L$ be a finite
separable extension of $K$ and let $S$ be an integral closure of $R$ in $L$. Then $S$ is
a Dedekind domain, finitely generated as $R$-module. Theorem~\ref{localglobal} implies that
$S$ can be generated as an $R$-algebra by no more than 2 elements. 
Is $S$ generated by 1 element? The answer is yes if $R$ is a discrete valuation ring.
The following example shows that the answer is no in general.

\begin{example}{\em
Let $K=\mathbb C(t)$. Take $L=K(\sqrt{t-1}, \sqrt[3]{t})$. Set $y=\sqrt{t-1}$, $x=\sqrt[3]{t}$, so 
$y^2=x^3-1$ and $L=\mathbb C(x,y)$ is the function field of the elliptic curve $y^2=x^3-1$. Then $S=\mathbb C[x,y]$ is the integral closure of $R=\mathbb C[t]$ in $L$, and $S$ 
can not be generated by one element as an $R$-algebra.

\begin{proof}
Note that $L$ is a Galois extension of $K$ of degree 6. The Galois group of $L/K$ is cyclic
generated by an automorphism $\tau$ such that $\tau(y)=-y$ and $\tau(x)= \lambda x$, where
$\lambda$ is a primitive cube root of 1. Note that the fixed field of $\tau^2$ is 
$K(y)=\mathbb C(y)$, $T_y=\mathbb C[y]$ is the integral closure of $R$ in $K(y)$, and $1,y$
is an integral basis of $K(y)/K$ with discriminant $4(t-1)$. 
The fixed field of $\tau^3$ is $K(x)=\mathbb C(x)$, $T_x=\mathbb C[x]$ is the integral closure of $R$ in $K(x)$, and $1, x, x^2$ is an integral basis of $K(x)/K$ with discriminant $-27t^2$.
As the discriminants of $K(x)/K$ and $K(y)/K$ are relatively prime in $R$, we conclude
that the integral closure of $R$ in $L=K(x,y)$ is $S=\mathbb C[x,y]$, and $1,y,x,x^2,xy,x^2y$ is an
integral basis of $L/K$ with discriminant $(4(t-1))^3(-27t^2)^2=6^6t^4(t-1)^3$ (see \cite[Proposition 2.11]{neuk}).
  
Suppose now that $S=R[h]$ for some $h\in S$. Then $1,h,h^2,h^3,h^4,h^5$ is an integral basis
of $L/K$ with discriminant $\displaystyle \prod_{0\leq i<j\leq 5}(\tau^i(h)-\tau^{j}(h))^2$.
Thus 
\begin{equation}\label{disc}
 \prod_{0\leq i<j\leq 5}(\tau^i(h)-\tau^{j}(h))^2=ct^4(t-1)^3 
\end{equation} 
for some constant $c\in \mathbb C$. 
We may choose $h$ of the form $h=a_1y+a_2x+a_3x^2+a_4yx+a_5yx^2$ for some
$a_1,\ldots, a_5\in R$. Then
\[ (\tau^i(h)-\tau^{i+3}(h))^2= 4(t-1)(a_1+a_4\lambda^ix+a_5\lambda^{2i}x^2)^2
\]
for $i=0,1,2$. It follows from (\ref{disc}) (interpreted as equality in $\mathbb C[x]$)
that $a_1(x^3)+a_4(x^3)x+a_5(x^3)x^2$ divides $x^{12}$. It follows that two of $a_1,a_4,a_5$
are 0 and the third is of the form $dt^i$ for some $i$. If $a_1=0$  or $a_1=dt^i$ with
$i>0$ then we would have $h=xg$ for some $g\in S$ and the discriminant of
$1,h,h^2,h^3,h^4,h^5$ would be divisible by $x^{30}=t^{10}$, a contradiction. It follows that
$a_1$ is constant and $a_4=a_5=0$. Now, when $0\leq i<j\leq 5$ have the same parity (there are six
such pairs $i,j$), then
\[ (\tau^i(h)-\tau^{j}(h))^2= (\lambda^i-\lambda^j)^2x^2(a_2+(\lambda^i+\lambda^j)a_3x)^2.
\]   
It follows that $\displaystyle x^{12}\prod(a_2(x^3)+(\lambda^i+\lambda^j)xa_3(x^3))^2$ divides 
$x^{12}$
in $\mathbb C[x]$. This can only happen if $a_3=0$ and $a_2$ is constant. We may then assume
that $h=x+uy$ for some constant $u\neq 0$. Then 
\[(h-\tau(h))^2(\tau^3(h)-\tau^4(h))^2 =[(1-\lambda)^2x^2-4u^2(x^3-1)]^2
\]
divides $x^{12}$ in $\mathbb C[x]$, a contradiction. 
\end{proof}
}
\end{example}

At this point we add the following to the notation and assumptions made at the beginning of this section:

{\sffamily{
\begin{itemize}
\item $R$ is the ring of integers in a number field $K$.

\item $n_{i,j}$ is the local capacity of $H_i$ at $P_{i,j}$.

\item  $m_{i,j}$ is the local index of $H_i$ at $P_{i,j}$.

\item $R/P$ is a finite field of order $q$, i.e. $R/P\cong \F_q$.

\item $f_{i,j}$ is the inertia degree of $P_{i,j}$ over $P$, so 
 $R_i/P_{i,j}\cong \F_{q^{f_{i,j}}}$.
 
\item $I(n,r)=\{(i,j): n_{i,j}=n\ \text{and} \ f_{i,j}m_{i,j}=r\}$.

\item $M(n,r)$ is the number of elements in $I(n,r)$

\item $\displaystyle A(n,r)=\prod_{(i,j)\in I(n,r)} A_{i,j}$.
\end{itemize}
Note that many of the above notions depend on the maximal ideal $P$ of $R$ and, if necessary, they will be treated as functions of $P$. For example, we will write $M(n,r,P)$, $f_{i,j}(P)$, $I(n,r,P)$
for $M(n,r)$, $f_{i,j}$, $I(n,r)$ respectively, if the dependence on $P$ needs to be indicated.
}}

By Theorem~\ref{simple}, $A_{i,j}=\Lambda_{i}/P_{i,j}^{e_{i,j}}\Lambda_i$ is a finite dimensional
primary $R/P$-algebra with simple quotient isomorphic to $\M_{n_{i,j}}(\F_{q^{f_{i,j}m_{i,j}}})$.
Thus $I(n,r)$ is the set of all pairs $(i,j)$ such that the simple quotient of $A_{i,j}$
is isomorphic to $\M_{n}(\F_{q^r})$ as $R/P=F_q$-algebras and $M(n,r)$ is the number of simple
factors of the semisimple algebra $(\Lambda/P\Lambda)/J(\Lambda/P\Lambda)$ which
are isomorphic to $\M_n(\F_{q^r})$.

\begin{theorem}\label{smallest}
Let $R$ be the ring of integers in a number field $K$ and let $\Lambda$ be a maximal
$R$-order in a finite dimensional semisimple $K$-algebra $H$. For each maximal ideal
$P$ of $R$ and positive integers $n,r$ let $M(n,r,P)$ be the number of simple
factors of the semisimple algebra $(\Lambda/P\Lambda)/J(\Lambda/P\Lambda)$ which
are isomorphic to $\M_n(\F_{q^r})$, where $q=\N(P)$. Let $h$ be the smallest
positive integer such that 
\begin{equation}\label{smallgen}
 M(n,r,P)\leq \frac{\g_h(n,q,r)}{r|\text{\rm PGL}_n(\mathbb F_{q^r})|}
\end{equation}
for every maximal ideal $P$ of $R$ and any positive integers $n,r$.
If $h\geq 3$ then $h$ is the smallest number of generators of $\Lambda$
as an $R$-algebra. If $h=2$ then the smallest number of generators is $2$ or $3$.
If $h=1$ then $\Lambda$ is commutative and the smallest number of generators is $1$ or $2$. 
\end{theorem}

%assume that
%$K$ is a number field and $R$ is the ring of integers of $K$. Furthermore, we will often assume
%that $\Lambda$ is free as $R$-module (this is needed when we use density as we have not yet extended
%it to the case when the the $R$-algebra is only projective over $R$).

%As before, we have  $H=\prod_{i} H_i$ and $\Lambda= \prod_{i} \Lambda_i$, where, for each $i$, $H_i$ is a simple separable $K$-algebra with center $K_i$, $\Lambda_i$ is a maximal $R_i$-order in $H_i$
%and $R_i$ is the integral closure of $R$ in $K_i$. Each $K_i$ is a number field and
%$R_i$ is the ring of integers of $K_i$. We have $PR_i=\prod_{j}P_{i,j}^{e_{i,j}}$, where $P_{i,j}$ %are maximal ideals of $R_i$ and $e_{i,j}=e_{i,j}(P)$ are positive integers. Thus
%\[ \Lambda/P\Lambda= \prod_{i} \Lambda_i/P\Lambda_i= \prod_{i,j}\Lambda_{i}/P_{i,j}^{e_{i,j}}\Lambda_i.
%\]
%Let $n_{i,j}$, $m_{i,j}$ be the local capacity and local index of $H_i$ at $P_{i,j}$ respectively.
%Let $R/P\cong \F_q$ and let $f_{i,j}$ be the inertia degree of of $P_{i,j}$ over $P$.
%Thus $R_i/P_{i,j}\cong \F_{q^{f_{i,j}}}$.
%Let $A_{i,j}=\Lambda_{i}/P_{i,j}^{e_{i,j}}\Lambda_i$. 
\begin{proof}
As $H$ is a separable $K$-algebra, we have $r_K(\Lambda)=1$ if $\Lambda$ is commutative
and $r_K(\Lambda)=2$ if $\Lambda$ is not commutative. By Theorem~\ref{localglobal}, it suffices
to prove that $h$ is the smallest number such that $\g_h(\Lambda, P)>0$ for all
maximal ideals $P$ of $R$. 

Consider a maximal ideal $P$ of $R$. By (\ref{modP}), we have
\[  \Lambda/P\Lambda= \prod A(n,r),\]
where the product is over all pairs $n,r$ such that $I(n,r)$ is non-empty.
Any simple quotient of $A(n,r)=\prod_{(i,j)\in I(n,r)} A_{i,j}$ is isomorphic to 
$\M_{n}(\F_{q^r})$. Thus, when $I(n,r)$ and $I(n_1,r_1)$ are distinct and non-empty, 
the algebras $A(n,r)$ and $A(n_1,r_1)$ have no isomorphic quotients.  By Proposition~\ref{productgen}, for any integer $k>0$ we have 
\begin{equation}\label{glambda}
\g_k(\Lambda, P)=\prod \g_k(A(n,r), R/P),
\end{equation}
where the product is over all pairs $n,r$ such that $I(n,r)$ is non-empty. 
By Theorem~\ref{primary}, we have
\[ \g_k(A(n,r), R/P)= \g_k(\M_{n}(\F_{q^r})^{M(n,r)}, \F_q)\prod_{(i,j)\in I(n,r)}g_k(A_{i,j},
A_{i,j}/J(A_{i,j})).\]
Recall now that $g_k(A_{i,j},A_{i,j}/J(A_{i,j}))= g_k(A_{i,j},A_{i,j}/J(A_{i,j})^2)g_k(A_{i,j}/J(A_{i,j})^2,A_{i,j}/J(A_{i,j}))$.
By Corollary~\ref{lifts}, $g_k(A_{i,j},A_{i,j}/J(A_{i,j})^2)=|J(A_{i,j})^2|^k$.
Define $c_{i,j}$ as follows:
\begin{equation}\label{cij}
c_{i,j}=\begin{cases} 1, & \text{if} \ m_{i,j}>1;\\
q^{-f_{i,j}}, & \text{if} \ m_{i,j}=1 \ \text{and} \  e_{i,j}>1; \\
0, & \text{if} \ m_{i,j}=1 \ \text{and} \  e_{i,j}=1. 
\end{cases}
\end{equation}
Then, by Theorem~\ref{bisimple} and Theorem~\ref{simple} we have 
\[ g_k(A_{i,j}/J(A_{i,j})^2,A_{i,j}/J(A_{i,j}))= |J(A_{i,j})/J(A_{i,j})^2|^k-c_{i,j}|J(A_{i,j})/J(A_{i,j})^2|.\]
Finally, $|J(A_{i,j})^t|=q^{n^2r(e_{i,j}m_{i,j}-t)}$ for all positive integers $t\leq e_{i,j}m_{i,j}$. Putting all this together we get
\begin{equation}\label{gennr}
\g_k(A(n,r), R/P)= \g_k(\M_{n}(\F_{q^r})^{M(n,r)}, \F_q)\prod_{(i,j)\in I(n,r)}q^{kn^2r(e_{i,j}m_{i,j}-2)}(q^{kn^2r}-c_{i,j}q^{n^2r}).
\end{equation}
By Theorem~\ref{formula}, $\g_k(\M_{n}(\F_{q^r})^{M(n,r)}, \F_q)>0$ if and only if
(\ref{smallgen}) holds for $M(n,r,P)=M(n,r)$. Thus, $h$ is the smallest
positive integer such that  $\g_k(\M_{n}(\F_{q^r})^{M(n,r)}, \F_q)>0$ for
all maximal ideals $P$ and all $n,r$ such that $M(n,r)>0$. If $h>1$, then
this is the same as $\g_h(A(n,r), R/P)>0$ for
all maximal ideals $P$ and all $n,r$ such that $M(n,r)>0$, so the theorem is true
in this case. If $h=1$, then $n=1$ whenever $M(n,r)>0$. It follows that $\Lambda$ 
is commutative and $c_{i,j}$ are always less than 1. Hence $\g_1(A(n,r), R/P)>0$ for
all maximal ideals $P$ and all $n,r$ such that $M(n,r)>0$ and the theorem holds
in this case as well. 
\end{proof}

When $h=1$ in Theorem~\ref{smallest}, then there are examples when the smallest number
of generators is $2$. In fact, finding such examples with additional assumption
that $H$ is a field has been an old topic in algebraic number theory. See the last section
of \cite{pleasant} for some examples. The next result shows that when $h=2$ the ambiguity
about the smallest number of generators can only happen when $H$ has a simple factor
of dimension 4 over its center.

\begin{theorem}\label{smallest1}
If in Theorem~\ref{smallest} $h=2$ and no simple factor of $H$ has dimension $4$ over its center
and $\Lambda$ is free as an $R$-module, then the smallest number of generators of $\Lambda$ as an $R$-algebra is $2$. If Conjecture~\ref{conjdense} is true then the assumption that $\Lambda$ is free
can be dropped. 
\end{theorem}
\begin{proof}
According to Theorem~\ref{density} and Conjecture~\ref{conjdense}, it suffice to prove that 
\begin{equation}\label{h=2}
\prod_{P \in \mspec R}\frac{\g_2(\Lambda, P)}{\N(P)^{2d}}>0,
\end{equation}
where $d$ is the dimension of $H$ over $K$. We have seen in the proof of Theorem~\ref{smallest}
that all the factors in (\ref{h=2}) are positive. In addition, for all but a finite
number of maximal ideals $P$, all the numbers $m_{i,j}$ and $e_{i,j}$ are equal to 1 and $n_{i,j}^2=d_i$ is the dimension of $H_i$ over $K_i$.  Let $C>d$
be such that all $P$ with $\N(P)>C$ have this property and $C>c(n,r,t)$ for all
$n,r,t$ such that $a\leq d$, $r\leq d$, $r\leq d$
(here $c(n,r,t)$ are the constants from Lemma~\ref{estimat}). 
Thus, when $\N(P)>C$ then
\begin{itemize}
\item $I(n,r,P)=\{(i,j): d_i=n^2\ \text{and}\ f_{i,j}=r\}$.

\item $A(n,r,P)=\M_{n}(\F_{q^r})^{M(n,r,P)}$, where $q=\N(P)$.

\item $\displaystyle d=\sum_{(n,r)} M(n,r,P)rn^2$. 
\end{itemize}
By (\ref{glambda}), (\ref{gennr}), and Lemma~\ref{estimat}, when $q=\N(P)>C$  then
\[\frac{\g_2(\Lambda, P)}{\N(P)^{2d}}=\prod_{(n,r)}\frac{\g_2(\M_{n}(\F_{q^r})^{M(n,r,P)}, \F_q)}{
q^{2M(n,r,P)rn^2}}\geq \prod_{(n,r)}\left(1-\frac{25M(n,r,P)}{q^2} \right)^{
M(n,r,P)}\geq
\]
\[\geq \left(1-\frac{25d}{q^2} \right)^{\sum_{(n,r)} M(n,r,P)  }\geq \left(1-\frac{25d}{q^2} \right)^d.\]
The result follows now from the well-known fact that the product
\[\prod_{P \in \mspec R,\  \N(P)>25d}\left(1-\frac{25d}{\N(P)^2} \right)\]
converges (i.e. it is positive).
\end{proof}

\begin{remark}
{\em Using Theorem~\ref{exactformula}(ii) it is not hard to show that if $H$ has a simple factor
of dimension 4 over its center then the product (\ref{h=2}) is indeed $0$}.
\end{remark}

\begin{example}{\em
Let $H=B^m$, where $B$ is a quaternion algebra over $\mathbb Q$ ramified at exactly the primes
$p_1<\ldots<p_t$ (and at infinity, if $t$ is odd). Let $\Lambda$ be a maximal order in $B$, so
$\Lambda^m$ is a maximal order in $H$. 
We have $M(1,2,p_i)=m$ and 
$M(n,r,p_i)=0$ if $n\neq 1$ or $r\neq 2$.
For every prime $p$ different from the ramified primes we have $M(2,1,p)=m$ and $M(n,r,p)=0$ if $n\neq 2$ or $r\neq 1$. Thus we are looking for the smallest $k$ such that
$\displaystyle m\leq \frac{\g_k(2,p)}{p(p^2-1)}$ for every $p$ unramified in $B$ and $m\leq \g_k(1,p_i,2)/2$ for $i=1,\ldots, t$. By Theorem~\ref{exactformula}, the first condition is equivalent to
\[ m\leq \frac{p^{2k-1}(p^k-1)(p^k-p)}{p^2-1}. \]
By Theorem~\ref{mobius}, the second condition can be stated as
\[ 2m\leq  p_i^{2k}-p_i^{k}. \]
If $2$ ramifies in $B$, i.e. if $p_1=2$, then all these inequalities hold if and only if
$m\leq 2^{k-1}(2^k-1)$. Thus there is $u\leq 6$ such that the minimal number of generators 
of $\Lambda^m$ is $2$ if and only if $m\leq u$, and it is $3$ if and only if
$u<m\leq 28$ and
it is $k>3$ iff $2^{k-2}(2^{k-1}-1)<m\leq 2^{k-1}(2^k-1)$. 
Suppose now that $2$ is unramified. Then our conditions will be satisfied iff
$m\leq 2^{2k-1}(2^k-1)(2^k-2)/3$ and $2m\leq  p_1^{2k}-p_1^{k}$. 
Note that $2^{2k-1}(2^k-1)(2^k-2)/3\leq 4^k(4^k-1)/3$. It follows that if $p_1>4$, i.e. 
if $3$ is also unramified, then $m\leq 2^{2k-1}(2^k-1)(2^k-2)/3$ implies the other condition.
Thus there is $u\leq 16$ such that the minimal number of generators 
of $\Lambda^m$ is $2$ if and only if $m\leq u$ and it is $3$ iff $u<m\leq 448$ and
it is $k>3$ iff $2^{2k-3}(2^{k-1}-1)(2^{k-1}-2)/3 < m\leq 2^{2k-1}(2^k-1)(2^k-2)/3$.
Finally suppose that $p_1=3$. Then the conditions are $ m\leq 2^{2k-1}(2^k-1)(2^k-2)/3$ and $m\leq  3^{k}(3^{k}-1)/2$. For $k=2$ this is equivalent to $m\leq 16$,  and for all $k\geq 3$ it
is equivalent to $m\leq  3^{k}(3^{k}-1)/2$. Thus there is $u\leq 16$ such that the minimal number of generators 
of $\Lambda^m$ is $2$ if and only if $m\leq u$ and it is $3$ iff $u<m\leq 351$ and
it is $k>3$ iff $3^{k-1}(3^{k-1}-1)/2 < m\leq 3^{k}(3^{k}-1)/2$.
In the introduction to \cite{kmp} we expressed a hope that knowing the smallest number of generators of $\Lambda^m$ for every positive integer $m$ may tell us a lot about the structure of $\Lambda$. This example shows that this is not the case.}
\end{example}

\begin{question}
What is $u$ in the above example?
\end{question}

\end{document}